\documentclass[a4paper, 11pt]{amsproc}
\usepackage{amsmath}    
\usepackage{amsthm}     
\usepackage{amssymb}    
\usepackage{mathtools}  
\usepackage{amstext}    

\usepackage{mathrsfs}   
\usepackage{stmaryrd}   
\usepackage{bm}         

\usepackage{amsrefs}    
\usepackage{ascmac}     
\usepackage{comment}    
\usepackage{here}       
\usepackage{time}       
\usepackage{fullpage}   
\usepackage{xcolor}     

\usepackage[utf8]{inputenc} 

\usepackage{hyperref}   
\hypersetup{
 colorlinks=true,       
 linkcolor=blue,        
 citecolor=blue,        
 filecolor=blue,        
 urlcolor=blue          
}

\usepackage[capitalize,nameinlink,noabbrev,nosort]{cleveref}

\DeclareFontEncoding{OT2}{}{}
\DeclareFontSubstitution{OT2}{cmr}{m}{l}
\DeclareFontFamily{OT2}{cmr}{\hyphenchar\font45 }
\DeclareFontShape{OT2}{cmr}{m}{l}{%
  <5><6><7><8><9>gen*wncyr%
  <10><10.95><12><14.4><17.28><20.74><24.88>wncyr10}{}
\DeclareMathAlphabet{\mathcyr}{OT2}{cmr}{m}{l}
\DeclareMathAlphabet{\mathcyb}{OT2}{cmr}{b}{l}
\SetMathAlphabet{\mathcyr}{bold}{OT2}{cmr}{b}{l}

\makeatletter
\@namedef{subjclassname@2020}{%
  \textup{2020} Mathematics Subject Classification}
\makeatother

\newtheorem{theoremcounter}{Theorem Counter}[section]
\theoremstyle{definition}

\newtheorem{remark}[theoremcounter]{Remark}

\theoremstyle{plain}
\newtheorem{lemma}[theoremcounter]{Lemma}

\newtheorem{corollary}[theoremcounter]{Corollary}

\newtheorem{theorem}[theoremcounter]{Theorem}

\numberwithin{equation}{section}

\allowdisplaybreaks[1]

\begin{document}

\title{A note on order estimates of the $q$-analogue of the Riemann zeta function}

\author{Hideki Murahara}
\address[Hideki Murahara]{The University of Kitakyushu,  4-2-1 Kitagata, Kokuraminami-ku, Kitakyushu, Fukuoka, 802-8577, Japan}
\email{hmurahara@mathformula.page}

\author{Tomokazu Onozuka}
\address[Tomokazu Onozuka]{Division of Mathematical Sciences, Department of Integrated Science and Technology, Faculty of Science and Technology, Oita University, 700 Dannoharu, Oita, 870-1192, Japan} 
\email{t-onozuka@oita-u.ac.jp}


\begin{abstract}
 At the first step of studying order estimates for the $q$-analogue of the Riemann zeta function, we estimate bounds for it on vertical lines for a fixed parameter $q$.
\end{abstract}


\maketitle

\section{Introduction}
For parameter $q\in(0,1)$ and complex variables $s$ and $t$, Kaneko, Kurokawa, and Wakayama \cite{KanekoKurokawaWakayama2003} defined the $q$-analogue of the Riemann zeta function
\begin{gather*}
 \zeta_q(s,t)
 \coloneqq\sum_{m=1}^\infty \frac{q^{mt}}{[m]_q^s},
\end{gather*}
where $[m]_{q}\coloneqq\frac{1-q^{m}}{1-q}$ is the $q$-integer.
They proved that this series converges absolutely for $s\in \mathbb{C}$ and $\Re(t)>0$, and that $\zeta_q(s,t)$ can be meromorphically continued to $\mathbb{C}^2$. 
All poles of $\zeta_q(s,t)$ are simple and are located at $t \in \mathbb{Z}_{\le 0}+2\pi i \mathbb{Z}/\log q\coloneqq\{ a+2\pi i b/\log q \mid a,b\in\mathbb{Z},a\le0 \}$.
Additionally, they established that $\lim_{q \nearrow 1} \zeta_q(s)=\zeta(s)$ for $s \in \mathbb{C} \setminus \{1\}$, where $\zeta_q(s)$ is defined as $\zeta_q(s, s-1)$.
All poles of $\zeta_q(s)$ are simple and are located at $s \in 1+2\pi i \mathbb{Z}/\log q$ and $s \in \mathbb{Z}_{\le 0}+2\pi i \mathbb{Z}_{\neq0}/\log q$.

In this paper, we investigate the growth of $\zeta_q(s,t)$ and $\zeta_q(s)$ as $|v| \to \infty$, where $s=\sigma+iv$ and $q$, $\sigma$, and $\Re(t)$ are fixed.
The main result is as follows.
\begin{theorem}\label{main}
Let $\varepsilon>0$, $q\in (0,1)$, and let $s=\sigma+iv$ and $t$ be complex numbers.
Fix $\varepsilon$, $q$, $\sigma$, and $\Re(t)$. 
Assume that $t$ satisfies $\inf_{r\in \mathbb{Z}_{\ge 0}} |1-q^{t+r}|>\varepsilon$. 
Then, for sufficiently large $|v|$, we have 
\begin{align*}
 \zeta_q(s,t)
 =
 \begin{cases}
  O_{}\left( 1 \right)&(\Re(t)>0),\\
  O_{}\left( |v| \right)&(\Re(t)=0),\\
  O_{}\left( 
   \exp\left(-\Re(t)(1+\pi/2) \cdot |v|\right) 
  \right)
  &(\Re(t)<0).
 \end{cases}
\end{align*}
Here the implicit constants depend on $\epsilon$, $q$, $\sigma$, and $\Re(t)$.
\end{theorem}

If we take $t=s-1$ in the above theorem, we get the following corollary.
\begin{corollary}\label{cor}
Let $\varepsilon>0$, $q\in (0,1)$, and $s=\sigma+iv$, where $s$ is a complex number.
Fix $\varepsilon$, $q$, and $\sigma$. 
Assume that $s$ satisfies $\inf_{r\in \mathbb{Z}_{\ge -1}} |1-q^{r+s}|>\varepsilon$. 
Then, for sufficiently large $|v|$, we have 
\begin{align*}
 \zeta_q(s)
 =
 \begin{cases}
  O_{}\left( 1 \right)&(\sigma>1),\\
  O_{}\left( |v| \right)&(\sigma=1),\\
  O_{}\left( 
   \exp\left(-(\sigma-1)(1+\pi/2) \cdot |v|\right) 
  \right)
  &(\sigma<1).
 \end{cases}
\end{align*}
Here the implicit constants depend on $\epsilon$, $q$, and $\sigma$.
\end{corollary}

\begin{remark}
We define \(\mu_q(\sigma)\) as  
\[
\mu_q(\sigma) =\mu_{q,\varepsilon}(\sigma) 
\coloneqq 
\inf 
\left\{ 
 \xi \in \mathbb{R} \mid \zeta_q(\sigma+iv) \ll_{q,\sigma,\varepsilon} |v|^\xi, \ \inf_{r \in \mathbb{Z}_{\ge -1}} |1-q^{\sigma+r+iv}| > \varepsilon
\right\}
\]
for fixed $q\in (0,1)$ and $\varepsilon>0$.
The corollary above implies that $\mu_q(\sigma)$ satisfies
\[
\mu_q(\sigma) \le 
\begin{cases} 
0 & (\sigma > 1), \\ 
1 & (\sigma=1),\\
\infty & (\sigma<1).
\end{cases}
\]
For the classical Riemann zeta function $\zeta(s)$, assuming the Lindel\"of hypothesis, we have
\[
\mu(\sigma)=
\begin{cases} 
0 & (\sigma\ge 1/2), \\ 
1/2-\sigma & (\sigma < 1/2).
\end{cases}
\]
One possible reason for the significant gap between the estimation of $\mu_q(\sigma)$ and $\mu(\sigma)$ is that $q$ is fixed in Corollary \ref{cor}. 
As $q$ approaches $1$, it is expected that $\mu_q(\sigma)$ will converge to $\mu(\sigma)$.
\end{remark}

\section{Proof of Theorem \ref{main}}\label{sec2}
\begin{lemma}\label{lem21}
Let $q\in(0,1)$ and $N$ be a positive integer.
For any $s\in\mathbb{C}$ and $t\in\mathbb{C}\setminus(\mathbb{Z}_{\le0}+2\pi i\mathbb{Z}/\log q)$, we have
\begin{align}\label{two_parts}
 (1-q)^{-s}\zeta_q(s,t)
 &=\sum_{m=1}^{N-1} \frac{q^{mt}}{(1-q^m)^s}+\sum_{r=0}^{\infty} \binom{r+s-1}{r}\frac{q^{N(t+r)}}{1-q^{t+r}}.
\end{align}
\end{lemma}
\begin{proof}
This lemma is obtained similarly as in \cite{KanekoKurokawaWakayama2003}*{PROPOSITION 1}.
The second term of the right-hand side is provided as
\begin{align*}
 \sum_{m=N}^{\infty} \frac{q^{mt}}{(1-q^m)^s}
 &=\sum_{m=N}^{\infty}\sum_{r=0}^{\infty} \binom{-s}{r}(-1)^rq^{m(t+r)}\\
 &=\sum_{r=0}^{\infty} \binom{r+s-1}{r}\sum_{m=N}^{\infty}q^{m(t+r)}\\
 &=\sum_{r=0}^{\infty} \binom{r+s-1}{r}\frac{q^{N(t+r)}}{1-q^{t+r}}
\end{align*}
for $\Re(t)>0$, and this series can be continued meromorphically to $\mathbb{C}^2$.
\end{proof}

\begin{lemma} \label{lem22}
Let $s=\sigma+iv$ and $r\in\mathbb{Z}_{\ge1}$. 
For fixed $\sigma$, if $|v|$ is sufficiently large, we have
\begin{align*}
 \binom{r+s-1}{r}
 &\ll_\sigma \frac{ |r+v|^{\sigma-1/2} e^{|\sigma+iv|} }{r^{1/2}|v|^{\sigma-1/2}}
 \exp\left(\frac{\pi}{2}|v|\right),
\end{align*}
where the implicit constant does not depend on $r$ and $v$.
\end{lemma}
\begin{proof}
By Stirling's formula, we have
\begin{align*}
 \binom{r+s-1}{r}&=\frac{\Gamma(r+s)}{r!\Gamma(s)}\\
 &\asymp\frac{|(r+s)^{r+s-1/2}e^{-r-s}|}{r!|s^{s-1/2}e^{-s}|}\\
 &\asymp\frac{|\sigma+r+iv|^{\sigma+r-1/2}\exp(-\arg(r+s)v)e^{-r-\sigma}}{r!|\sigma+iv|^{\sigma-1/2}\exp(-\arg(s)v)e^{-\sigma}},
\end{align*}
where the argument is considered within the range from $-\pi$ to $\pi$.
For any $r\ge1$, we have 
\begin{align*}
 \exp(-\arg(r+s)v+\arg(s)v)
 \ll_\sigma \exp\left(\frac{\pi}{2}|v|\right).
\end{align*}
Using the asymptotic formula 
$r!\asymp r^{r+1/2} e^{-r}$,
we find that
\begin{align*}
 \binom{r+s-1}{r}
 &\ll_\sigma \frac{ e^{-r} |\sigma+r+iv|^{r} |r+v|^{\sigma-1/2} }{r!|v|^{\sigma-1/2}}
 \exp\left(\frac{\pi}{2}|v|\right)\\
 &\ll_\sigma \frac{ |\sigma+r+iv|^{r} |r+v|^{\sigma-1/2} }{r^{r+1/2}|v|^{\sigma-1/2}}
 \exp\left(\frac{\pi}{2}|v|\right).
\end{align*}
Next, observing that
\begin{align*}
 \frac{ |\sigma+r+iv|^{r} }{ r^{r} }
 \le \left(1+ \frac{|\sigma+iv|}{r} \right)^{r}
 \le e^{|\sigma+iv|},
\end{align*}
we deduce
\begin{align*}
 \binom{r+s-1}{r}
 &\ll_\sigma \frac{ |r+v|^{\sigma-1/2} e^{|\sigma+iv|} }{r^{1/2}|v|^{\sigma-1/2}}
 \exp\left(\frac{\pi}{2}|v|\right).
\end{align*}
This completes the proof.
\end{proof}

Now, we estimate the second term of Lemma \ref{lem21}.
\begin{lemma} \label{lem23}
Let $q\in(0,1)$, $N$ be a positive integer, $s=\sigma+iv$, and 
$t\in\mathbb{C}\setminus(\mathbb{Z}_{\le0}+2\pi i\mathbb{Z}/\log q)$.
For fixed $q$ and $\sigma$, if $|v|$ is sufficiently large, we have
\begin{align*}
 &\sum_{r=0}^{\infty} 
 \binom{r+s-1}{r}
 \frac{q^{N(t+r)}}{1-q^{t+r}}\\
 &\ll_{\sigma}
  \frac{q^{N\Re(t)}}{|1-q^{t}|}
  +\frac{q^{N\Re(t)} 
  e^{|\sigma+iv|}}{\inf_{r\in\mathbb{Z}_{\ge1}}|1-q^{t+r}|}
  \exp\left(\frac{\pi}{2}|v|\right)
 \cdot
 \left\{
  \frac{q^{N}}{1-q^{N}}
  +
  |v|^{-1/2}q^{N|v|}
  \left(
   1+\frac{1}{-N\log q}
  \right)
 \right\}.
\end{align*}
\end{lemma}
\begin{proof}
From Lemma \ref{lem22}, we have
\begin{align*}
 &\sum_{r=0}^{\infty} \binom{r+s-1}{r}\frac{q^{N(t+r)}}{1-q^{t+r}}\\
 &\ll_\sigma 
 \left|\frac{q^{Nt}}{1-q^{t}}\right|
 +\sum_{r=1}^{\infty} 
 \left|\frac{q^{N(t+r)}}{1-q^{t+r}}\right|
 \frac{ |r+v|^{\sigma-1/2} e^{|\sigma+iv|} }{r^{1/2}|v|^{\sigma-1/2}}
 \exp\left(\frac{\pi}{2}|v|\right)\\
 &\ll_{\sigma} 
 \frac{q^{N\Re(t)}}{|1-q^{t}|}
 +\frac{q^{N\Re(t)} e^{|\sigma+iv|}}{\inf_{r\in\mathbb{Z}_{\ge1}}|1-q^{t+r}|} 
 \exp\left(\frac{\pi}{2}|v|\right)
 \left\{
  \sum_{r\le |v|}
  \frac{ q^{Nr} }{ r^{1/2} }
  +
  \frac{1}{|v|^{\sigma-1/2}}
  \sum_{r>|v|}
  r^{\sigma-1} q^{Nr}
 \right\}.
\end{align*}
Here, 
\begin{align*}
 \sum_{r<|v|}
 \frac{ q^{Nr} }{ r^{1/2} }
 \le
 \sum_{r=1}^{\infty}
 q^{Nr}
 =\frac{q^{N}}{1-q^{N}}.
\end{align*}
For sufficiently large $|v|$, we have 
\begin{align*}
 \sum_{r>|v|}
 r^{\sigma-1} q^{Nr}
 &\le
 |v|^{\sigma-1} q^{N|v|}
 +\int_{|v|}^\infty
 x^{\sigma-1} q^{Nx} \,dx\\
 &=
 |v|^{\sigma-1} q^{N|v|}
 +\frac{1}{-N\log q} |v|^{\sigma-1} q^{N|v|}
 +(\sigma-1)
 \int_{|v|}^\infty
 \frac{x^{\sigma-2} q^{Nx}}{-N\log q} \,dx.
\end{align*}
Note that we can take sufficiently large $|v|$ which does not depend on $N$.
Since
\begin{align*}
 \int_{|v|}^\infty
 x^{\sigma-2} q^{Nx} \,dx
 \le
 |v|^{\sigma} q^{N|v|} 
 \int_{|v|}^\infty
 x^{-2} \,dx 
 =|v|^{\sigma-1} q^{N|v|}, 
\end{align*}
we have
\begin{align*}
 \sum_{r>|v|}
 r^{\sigma-1} q^{Nr}
 &\ll_\sigma
 |v|^{\sigma-1} q^{N|v|}
 \left(
  1+\frac{1}{-N\log q}
 \right).
\end{align*}
This finishes the proof. 
\end{proof}

We set
\begin{align*}
 N
 =\left\lfloor 
 -\frac{1+\pi/2}{\max\{\Re(t),0\}+1} \cdot \frac{|v|}{\log q}
 \right\rfloor.
\end{align*}
Note that
\begin{align*}
 q^N
 \asymp_{q}
 \exp\left(-\frac{1+\pi/2}{\max\{\Re(t),0\}+1} \cdot |v|\right).
\end{align*}

\begin{proof}[Proof of Theorem \ref{main}]
The first term of \eqref{two_parts} can be estimated as
\begin{align*}
 \sum_{m=1}^{N-1} \frac{q^{mt}}{(1-q^m)^s}
 &=O_{}\left(\max\{(1-q)^{-\sigma},(1-q^{N-1})^{-\sigma}\}
  \sum_{m=1}^{N-1} q^{m\Re(t)}\right)\\
 &=\max\{(1-q)^{-\sigma},(1-q^{N-1})^{-\sigma}\}\times\begin{cases}
  O_{q,\Re(t)}\left( 1 \right)&(\Re(t)>0),\\
  O_{}\left( N \right)&(\Re(t)=0),\\
  O_{q,\Re(t)}\left( q^{N\Re(t)}  \right)&(\Re(t)<0)
 \end{cases}\\
 &=\begin{cases}
  O_{\sigma,q,\Re(t)}\left( 1 \right)&(\Re(t)>0),\\
  O_{\sigma,q}\left( |v| \right)&(\Re(t)=0),\\
  O_{\sigma,q,\Re(t)}\left( \exp\left(-\Re(t)(1+\pi/2) \cdot |v|\right)  \right)&(\Re(t)<0).
 \end{cases}
\end{align*}
From Lemma \ref{lem23}, the second term of \eqref{two_parts} can be estimated as
\begin{align*}
 &\sum_{r=0}^{\infty} \binom{r+s-1}{r}\frac{q^{N(t+r)}}{1-q^{t+r}}\\
 &
 \ll_{\sigma,q,\Re(t),\varepsilon}
 \exp
   \left(
    -(1+\pi/2)
    \frac{\Re(t)}{\max\{\Re(t),0\}+1}
    \cdot |v|
   \right)\\
 &\quad
  +\exp
   \left(
    -(1+\pi/2)
    \frac{\Re(t)}{\max\{\Re(t),0\}+1}
    \cdot |v|
   \right)
  e^{|v|} 
  \exp\left(\frac{\pi}{2}|v|\right)
 \cdot
  \exp\left(-\frac{1+\pi/2}{\max\{\Re(t),0\}+1} \cdot |v|\right)\\
 &\ll_{\sigma,q,\Re(t),\varepsilon}
  \exp
   \left(
    \frac{1+\pi/2}{\max\{\Re(t),0\}+1}
    (\max\{\Re(t),0\}-\Re(t))
    \cdot |v|
    \right).
\end{align*}
From Lemma \ref{lem21}, this finishes the proof.
\end{proof}

\section*{Acknowledgments}
This work was supported by JSPS KAKENHI Grant Numbers JP22K13897 (Murahara) and JP19K14511 (Onozuka). 

\bibliographystyle{amsalpha}
\bibliography{References} 
\end{document}